\documentclass[12pt]{amsart}
\usepackage[headings]{fullpage}

\usepackage{amsmath,amssymb,amsthm}
\usepackage{graphicx}
\usepackage{enumerate}
\usepackage{url}
\usepackage{slashed}
\usepackage{bm}
\usepackage{array}

\usepackage{tikz}
\usepackage{tikz-cd}
\usepackage[german,english]{babel}

\usetikzlibrary{arrows.meta}

\usepackage[bookmarks=true,%
    colorlinks=true,%
    linkcolor=blue,%
    citecolor=blue,%
    filecolor=blue,%
    menucolor=blue,%
    urlcolor=blue,%
    breaklinks=true]{hyperref}

\usepackage[all]{xy}
\usepackage{verbatim}

\newtheorem{thm}{Theorem}[]
\newtheorem*{thm*}{Theorem}
\newtheorem{lem}[thm]{Lemma}
\newtheorem{prop}[thm]{Proposition}
\newtheorem{cor}[thm]{Corollary}

\newtheorem{ex}[thm]{Example}
\newtheorem{rem}[thm]{Remark}
\newtheorem{defn}[thm]{Definition}

\newcommand{\param}{{\mathchoice{\mkern1mu\mbox{\raise2.2pt\hbox{$
\centerdot$}}
\mkern1mu}{\mkern1mu\mbox{\raise2.2pt\hbox{$\centerdot$}}\mkern1mu}{
\mkern1.5mu\centerdot\mkern1.5mu}{\mkern1.5mu\centerdot\mkern1.5mu}}}

\begin{document}
\title{The diameter function is a topological Morse function}
\author{Ingrid Irmer, Bhola Nath Saha}
\address{SUSTech International Center for Mathematics\\
Southern University of Science and Technology\\Shenzhen, China
}
\address{Department of Mathematics\\
Southern University of Science and Technology\\Shenzhen, China
}
\email{ingridmary@sustech.edu.cn}
\address{Department of Mathematical Science\\
Indian Institute of Science Education and Research Berhampur}

\email{sahabholanath497@gmail.com}

\today

\begin{abstract}
Schmutz Schaller developed techniques for studying Teichm\"uller space using the systole function. These were presented in  \cite{SchmutzMorse}, \cite{SchmutzVoronoi} as a hyperbolic analogue of Voronoi's theory of quadratic forms in the theory of Euclidean lattice packings and coverings \cite{VoronoiPDQF}. It is known that the packing density function on Euclidean space is a topological Morse function, \cite{TMFAsh}, and the same is true of the systole function on Teichm\"uller space, \cite{Akrout}, \cite{SchmutzMorse}. The study of hyperbolic packing and covering problems is technical, for example, the density depends on the scale, and very little is known about optimisers of the density \cite{tóth2022ballpackingshyperbolicspace}. At least in the Euclidean setting, there are also fewer techniques available for studying sphere covering as opposed to sphere packing problems, as the covering problems seem to have less discernible structure. One approach to studying efficient circle coverings in the hyperbolic plane is to study the critical points of the diameter function on Teichm\"uller space. This paper shows that the diameter function on Teichm\"uller space is a topological Morse function. As a mapping class group-equivariant topological Morse function, critical points of the diameter function are related to the homology of moduli space. It would seem that for small genus, the systole function and diameter function have a larger proportion of common critical points than at higher genus.
\end{abstract}

\maketitle


\section{Introduction and history of the problem}
\label{intro}

Let $\mathcal{T}_{g}$ be the Teichm\"uller space of closed, connected surfaces of genus $g\geq 2$. The diameter function $f_{\mathrm{dia}}:\mathcal{T}_{g}\rightarrow \mathbb{R}_{+}$ is the function on Teichm\"uller space $\mathcal{T}_{g}$ that takes a point $x$ to the diameter of the marked hyperbolic surface corresponding to $x$.  It is invariant under the action of the mapping class group and bounded from below.  
\begin{prop}
\label{PWS}
$f_{\mathrm{dia}}$ is a continuous piecewise smooth function $\mathcal{T}_{g}\rightarrow \mathbb{R}_{+}$.
\end{prop}

Topological Morse functions were first defined in \cite{Morse}, and will be defined later in Section \ref{secdefns}. They generalise most of the properties of smooth Morse functions, and often appear in the context of sphere packings or Voronoi decompositions of manifolds.

 In order to compare with the well-known topological Morse function given by the systole function $f_{\mathrm{sys}}$ on $\mathcal{T}_{g}$, we choose to work with the function $-f_{\mathrm{dia}}$ in place of $f_{\mathrm{dia}}$. Like the systole function, $-f_{\mathrm{dia}}$, takes its maximal value somewhere in the thick part of $\mathcal{T}_{g}$, and decreases towards the noded surfaces in the metric completion of $\mathcal{T}_{g}$ with respect to the Weil-Petersson metric.
\begin{thm}
\label{mainthm}
$f_{\mathrm{dia}}$ is a topological Morse function on $\mathcal{T}_{g}$.
\end{thm}

By \cite{ConvexTMF}, proving Theorem \ref{mainthm} amounts to proving local finiteness and convexity properties that will be explained below.

For a curve $c$, denote by $L(c):\mathcal{T}_{g}\rightarrow \mathbb{R}$ the function taking $x$ to the length of $c$ on the corresponding marked hyperbolic surface.

Both $f_{\mathrm{sys}}$ and $-f_{\mathrm{dia}}$ will be shown to be variants of the Min-type functions defined in \cite{Hyam}. In Subsection \ref{rigiditysec} $n$-pods and their lengths will be defined, where it will be shown that $n$-pods can be used to study $f_{\mathrm{dia}}$ in much the same way as the lengths of closed curves are used to study $f_{\mathrm{sys}}$.  To highlight the similarities between the two functions, note that the systole function, $f_{\mathrm{sys}}$ is defined to be the function $\mathcal{T}_{g}\rightarrow \mathbb{R}$ whose value at $x\in \mathcal{T}_{g}$ is given by 
\begin{equation*}
\inf\{L(c)(x)\ |\ c\text{ is a closed geodesic on }S_{g}\},
\end{equation*}
where $S_{g}$ is the marked hyperbolic surface corresponding to the point $x\in\mathcal{T}_{g}$. Similarly, the diameter function, $f_{\mathrm{dia}}$ is defined to be the function $\mathcal{T}_{g}\rightarrow \mathbb{R}$ whose value at $x\in \mathcal{T}_{g}$ is given by 
\begin{equation*}
\text{-}f_{\mathrm{dia}}=\inf\{\text{-}L(P_{n})(x)\ |\ P_{n}\text{ is an }n\text{-pod on }S_{g}\}.
\end{equation*}
where the length $L(P_{n})$ of an $n$-pod $P_{n}$ will be defined in Subsection \ref{rigiditysec}.

One advantage of using $f_{\mathrm{dia}}$ in place of $f_{\mathrm{sys}}$ is that $f_{\mathrm{dia}}$ is directly related to the study of circle coverings of the hyperbolic plane, which appear to have many distinct properties from the more intensely studied circle packings of Euclidean space, see for example \cite{packingandcovering}. Another reason is that $f_{\mathrm{dia}}$ does not have the type of pathological behaviour described in Remark 4.4 of \cite{FillSpan}, and hence appears to be closer to a perfect Morse function on $\mathcal{T}_{g}$ than $f_{\mathrm{sys}}$. As a consequence of an argument credited to Bavard, (see, for example, Proposition 10 of \cite{DiffTop}), as $f_{\mathrm{sys}}$ and $f_{\mathrm{dia}}$ are both equivariant topological Morse functions, any isolated fixed point of the action of the mapping class group on $\mathcal{T}_{g}$ is a critical point of both. Many important examples of critical points of $f_{\mathrm{sys}}$ are known to be isolated fixed point sets of finite subgroups of triangle groups, for example the Hurwitz surfaces. This implies that in genus 2 for example, at least 3 out of the 4 critical points of $f_{\mathrm{sys}}$ are also critical points of $f_{\mathrm{dia}}$, where the critical points of $f_{\mathrm{sys}}$ in genus 2 are given in Theorem 44 of \cite{SchmutzMorse}. For every larger genus, there still exist critical points in common, for example the infinite family of critical points given in Theorem 36 of \cite{SchmutzMorse}, but the proportion of critical points obtained as isolated fixed points of the action of the mapping class group on $\mathcal{T}_{g}$ is smaller, so less is known.

We define the diameter locus $\mathcal{L}_{d}(x)$ of a marked hyperbolic surface $S_{g}$ consisting of the set of pairs of points in $S_{g}$ whose distance realise the diameter of $S_{g}$. It seems as though  $\mathcal{L}_{d}(x)$ is a classical object on which there should exist some mathematical literature. However, we were not able to find any references to the subject, so have derived the results needed. 

\begin{prop}
\label{DLL}
For $g\geq 2$ and every $x\in \mathcal{T}_{g}$, the diameter locus is a smoothly embedded finite graph on $S_{g}$ of dimension at most one.
\end{prop}

Proposition \ref{DLL} strongly suggests the existence of an algorithm for computing $f_{\mathrm{dia}}$ on $\mathcal{T}_{g}$, but so far very little is known, \cite{Stepanyants}.

\section{Definitions}
\label{secdefns}
Where there is no possibility of confusion, the symbol $S_{g}$ will be used to denote either the marked hyperbolic surface of genus $g$ corresponding to a point $x\in \mathcal{T}_{g}$, or to the topological surface of genus $g$. A curve is assumed to be an isotopy class of simple closed loops on $S_{g}$. The length of the curve is then the length of the unique geodesic representative of the isotopy class.



Let $ M $ be a topological manifold of dimension $n$ and $ f: M \to \mathbb{R} $ be a continuous function. A point $ p \in M $ is called a \emph{regular point} of the function $ f $ if there exists a coordinate chart $ (U,\phi) $ around $p$ such that $ f $ coincides with one of the coordinate functions, that is, 
$$f=x_{i},$$
for some $ i \in \{1,\dots, n\}, $ where $ \phi = (x_1, \dots, x_n) $. A point is called a \emph{critical point} if it is not a regular point. A critical point is non-degenerate if there exists a coordinate chart $ (U,\phi) $ around $p$ such that for all $z\in U$,
$$f(z) -f(p)= \sum \limits_{i=1}^{j}-x_i^2+\sum\limits_{i=j+1}^{n} x_i^{2}.$$
Here $j$ is the \emph{index} of the critical point.

\begin{defn}[Topological Morse function]
	A function is called a topological Morse function if every point of $M$ is either regular or a nondegenerate critical point.
\end{defn}


For the remainder of this paper, it will be assumed that $g\geq 2$, i.e. we will be studying the diameter of hyperbolic surfaces.

\begin{defn}[Diameter Locus $\mathcal{L}_{d}(x)$]
A point $x\in \mathcal{T}_{g}$ determines a marked hyperbolic surface $S_{g}$. The diameter locus $\mathcal{L}_{d}(x)$ is the locus of pairs of points $p_{1}, p_{2}$ in $S_{g}$ for which the distance $d(p_{1}, p_{2})$ in $S_{g}$ with respect to the hyperbolic metric on $S_{g}$ is equal to the diameter of $S_{g}$.
\end{defn}

As we are working with compact surfaces, there must be at least one pair of points in the diameter locus, so $\mathcal{L}_{d}(x)$ is never empty. When $x$ is a typical point in $\mathcal{T}_{g}$, we expect that  $\mathcal{L}_{d}(x)$ is zero dimensional. The next example is interesting as the surface has a lot of symmetry, and $\mathcal{L}_{d}$ appears to be maximal in some sense.

\begin{ex}[The diameter locus of the Bolza surface]
The Bolza surface $X_2$ is a hyperbolic surface of genus two. It is obtained by identifying the opposite sides of a regular hyperbolic octagon, whose edges all meet at angle $\frac{\pi}{4}$. As an example of a Hurwitz surface, is known to be the most symmetric hyperbolic surface of genus 2 and is the unique global maxima of the systole function in genus two. As is the case for all genus 2 surfaces, $X_2$ is a hyperelliptic surface and we will show that the fixed points of the hyperelliptic involution are elements of the diameter locus. It was shown in \cite{saha2023} that the hyperelliptic involution of the surface can be described as a rotation of the regular octagonal fundamental domain described above by an angle $\pi$ about the center of the octagon. The fixed points are then given by the center, the midpoints of the sides and the vertices of the fundamental domain. 
	
	In \cite{Stepanyants}, it was shown that the diameter of the Bolza surface is realised by the center and the vertex of the octagon and the length of the geodesic arc joining these points is the same as the diameter (the dashed segments in Figure \ref{fig:2}$(i)$). 
	
	It is easy to see that the arc $v_1Ov_5$ is a simple closed geodesic on the surface of length twice the diameter. We say that two points $P$ and $Q$ on this geodesic are diametrically opposite if they subdivide the geodesic in two arcs of equal lengths. We will show that if $P$ and $Q$ are not fixed points, then the distance between $P$ and $Q$ in the surface is less than the diameter. Our strategy to prove the claim is as follows: first, we find an isometry $f$ of the Poincare unit disc $\mathbb{D}$ which is an element of the Fuchsian group $\Gamma$ such that $\mathbb{D}/\Gamma \cong X_2$; then we show that the distance between $f(P)$ and $Q$ in $\mathbb{D}$ is less than the diameter of the surface.
	
	Let $f$ be the isometry that maps the side $v_4v_5$ to the side $v_8v_1$. Then the map $f: \mathbb{D} \to \mathbb{D}$ is given by
	$$f(z)=\frac{z+\tanh\left( \frac{s}{2} \right)}{\tanh\left( \frac{s}{2} \right)z+1},$$
	where $s=2\cosh^{-1}(\sqrt{2}+1)$ is the length of the sides of the octagon. Now let the distance of the point $P$ from the center $O$ be $t$ in $\mathbb{D}$. Then the distance of the point $Q$ from $O$ is $d-t$, where $d=\cosh^{-1}(3+2\sqrt{2})$ is the diameter of $X_2$. The coordinates of the points $P$ and $Q$ are therefore given by $-\tanh{\left( \frac{t}{2}\right)}  e^{\frac{i \pi}{8}}$ and $\tanh{\left( \frac{d-t}{2}\right)}  e^{\frac{i \pi}{8}}$, respectively. Now we can find the distance between the points $f(P)$ and $Q$ using the formula
	$$d(z,w)= \frac{|1-z\bar{w}|+ |z-w|}{|1-z\bar{w}|- |z-w|}, z,w \in \mathbb{D},$$
	and see (using the Matlab or Mathematica) that the distance is less than $d$ for $0<t<d$.
	
	To see the complete diameter locus of the Bolza surface, first we study the isometry group of $X_2$. By joining the midpoints of consecutive sides and the midpoints of opposite sides of the octagon with geodesic segments, the octagon is subdivided into $16$ equilateral triangles, each having interior angle $\pi/4$. The maps $\phi$, $\psi$, and $\eta$, defined respectively as a rotation by angle $\pi/4$ about the center of the octagon, a rotation by angle $\pi/4$ about the center of a triangle, and a reflection across a side of a triangle, are isometries of the Bolza surface (see Figure~\ref{fig:2}$(ii)$). For more details, see Section 2.1 \cite{saha2023}. 
	
	Applying these isometries, we get the full diameter locus, as depicted in Figure \ref{fig:diamlocus}.
\end{ex}

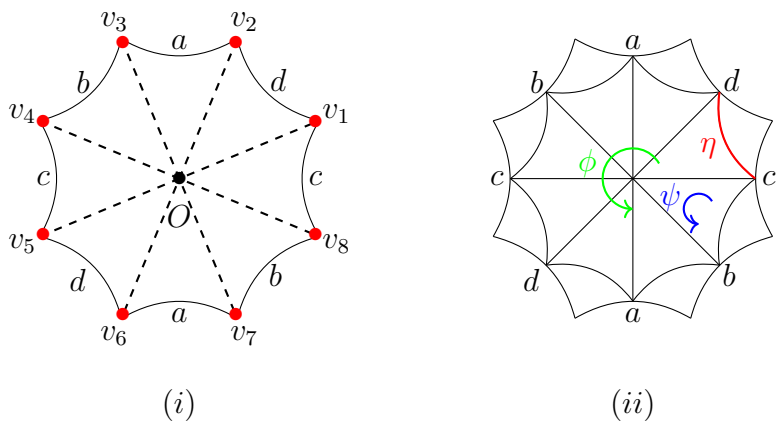
\begin{figure}[htbp]
	\centering
	\begin{tikzpicture}
		\begin{scope}
			\foreach \x in {1,2,...,8}
			{
				\draw [bend left] (-22.5+\x*45:2) to (22.5+\x*45:2);
				
			}
			
			\foreach \x in {1,2,3,4}
			{
			 \draw[thick, dashed] (22.5+\x*45:2) to (180+22.5+\x*45:2);
			}
			
			\foreach \x in {1,2,...,8}
			{
				\draw[red]  (22.5+\x*45:1.95) node {$\bullet$};
			}
			
			\draw (0,0) node {$\bullet$};
			
			\draw (0:1.8) node {$c$} (180:1.8) node {$c$} (45:1.85) node {$d$} (225:1.9) node {$d$} (90:1.8) node {$a$} (270:1.8) node {$a$} (135:1.8) node {$b$} (315:1.8) node {$b$};
			
			\foreach \x in {1,2,...,8}
			 {
			 	\draw  (-22.5+\x*45:2.25) node {$v_{\x}$};
			 }
		 	
		 	\draw (0,-.5) node {$O$};
			
			\draw  (0,-3) node {$(i)$};
		\end{scope}
		
		\begin{scope}[shift={(6,0)}]
			\foreach \x in {1,2,...,8}
			{
				\draw [bend left] (-22.5+\x*45:2) to (22.5+\x*45:2);
				\draw[bend left] (45+\x*45:1.62) to (90+\x*45:1.62);
			}
			\foreach \x in {1,2,3,4}
			{
				\draw (45+\x*45:1.62) to (225+\x*45:1.62);
			}
			\draw[->,thick, green] ({.4*cos(30)},{.4*sin(30)}) arc
			[
			start angle=30,
			end angle=270,
			x radius=.4cm,
			y radius =.4cm
			] ;
			
			\draw[->,thick, blue] ({.7+.4*cos(30)},{-.5+.4*sin(30)}) arc
			[
			start angle=30,
			end angle=270,
			x radius=.2cm,
			y radius =.2cm
			];
			
			\draw[bend left, red, thick] (0:1.62) to (45:1.62);
			
			\draw [green, thick] (-.6,.2) node {$\phi$};
			\draw[thick, blue] (.5,-.25) node {$\psi$};
			\draw[red, thick] (1,.4) node {$\eta$};
			
			\draw (0:1.8) node {$c$} (180:1.8) node {$c$} (45:1.85) node {$d$} (225:1.9) node {$d$} (90:1.8) node {$a$} (270:1.8) node {$a$} (135:1.8) node {$b$} (315:1.8) node {$b$};
			
			\draw  (0,-3) node {$(ii)$};
		\end{scope}
	\end{tikzpicture}
	\caption{$(i)$ The thick dots represent the points realising the diameter and the dashed arcs are the geodesic arcs with length equal to the diameter of the surface. $(ii)$ The isometries of the Bolza surfaces.}
	\label{fig:2}
\end{figure}

\begin{figure}[htbp]
	\centering
	\begin{tikzpicture}[scale=2]
		\foreach \x in {1,2,...,8}
		{
			\draw [bend left] (-22.5+\x*45:2) to (22.5+\x*45:2);
			
		}
		
		\foreach \x in {1,2,...,8}
		{
			\draw[dashed] (0,0) to (22.5+\x*45:2);
			\draw[red]  (22.5+\x*45:2) node {$\bullet$};
		}
		
		\draw (0,0) node {$\bullet$};
		
		\draw (0:1.8) node {$c$} (180:1.8) node {$c$} (45:1.85) node {$d$} (225:1.9) node {$d$} (90:1.8) node {$a$} (270:1.8) node {$a$} (135:1.8) node {$b$} (315:1.8) node {$b$};

		\foreach \x in {1,2,...,8}
		{
			\draw[red, dashed, bend left] (\x*45:1.62) to (90+\x*45:1.62);
		}
	
		\foreach \x in {1,2,...,8}
		{
			\draw [bend left, cyan, thick, dashed] (\x*45:1.62) to (33.75+\x*45:1.77);
		}

			\foreach \x in {1,2,...,8}
		{
			\draw [bend right, thick, cyan, dashed] (\x*45:1.62) to (-33.75+\x*45:1.7);
		}
	
		\foreach \x in {1,2,...,8}
		{
			\draw [blue] (\x*45:1.62) node {$\bullet$};
		}
	\end{tikzpicture}
	\caption{The diameter locus of the Bolza surface is given by the blue, red and black dots.}
	\label{fig:diamlocus}
\end{figure}
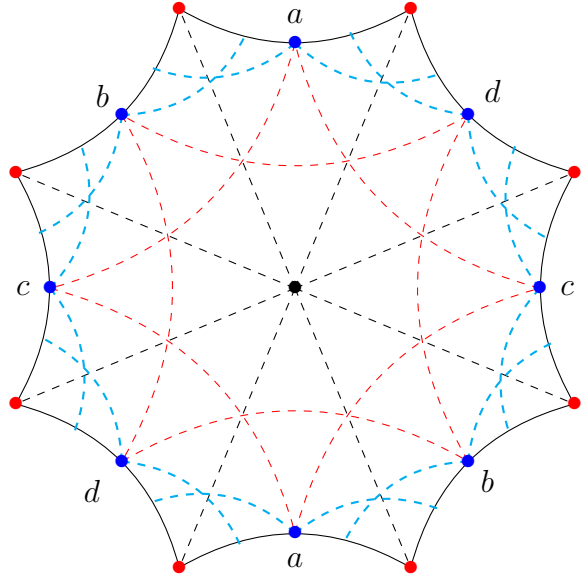

\begin{rem}[Some simple observations about $\mathcal{L}_{d}(x)$]
\label{basics}
Suppose $(p, p_{1})$, $(p, p_{2}), \ldots, (p,p_{k})$ are pairs of points in $\mathcal{L}_{d}(x)$. Then the geodesic arcs of length $f_{\mathrm{dia}}(x)$ joining these pairs cannot cross. This is because if $a_{i}$ and $a_{j}$ were arcs connecting $(p, p_{i})$ and $(p,p_{j})$, $1\leq i\leq j\leq k$ that did cross, cutting and pasting would give a piecewise-smooth geodesic arc with geodesic representative of length less than $f_{\mathrm{dia}}(x)$ joining $p$ and $p_{i}$ or $p$ and $p_{j}$. This contradicts the assumption that $p$ is distance $f_{\mathrm{dia}}(x)$ from $p_{i}$ and $p_{j}$. 

Another simple observation is that for every $i=1, \ldots, k$, there must be at least 3 geodesic arcs between $p$ and $p_{i}$ of length $f_{\mathrm{dia}}(x)$; if there were only 1 or 2, it would be possible to find a point near $p_{i}$ at distance greater than $f_{\mathrm{dia}}(x)$ from $p$. As $S_{g}$ is hyperbolic, there are no bigons, so any two such arcs joining the same pair in $\mathcal{L}_{d}(x)$ can be concatenated to give a simple, homotopically nontrivial closed curve on $S_{g}$. The set of all curves obtained in this way has the property that any two elements have pairwise geometric intersection number at most two, and length less at most $2f_{\mathrm{dia}}(x)$.
\end{rem}

\begin{lem}
\label{curvelength}
Let $p_{1}$ and $p_{2}$ be two points on $S_{g}(x)$ with $d(p_{1}, p_{2})=f_{\mathrm{dia}}(x)$, and let $C$ be the set of all curves obtained by concatenating a pair of distinct geodesics arcs of length $f_{\mathrm{dia}}(x)$ from $p_{1}$ to $p_{2}$. Then the cardinality of $C$ is finite.
\end{lem}
\begin{proof}
Every element of $C$ has length less than $2f_{\mathrm{dia}}(x)$. Since it is known that for any $L>0$ there are at most finitely many curves on $S_{g}$ of length less than $L$ on $S_{g}$, see for example \cite{M}, the result follows.
\end{proof}

\begin{defn}[$n$-pod]
An $n$-pod is an embedded graph in $S_{g}$ with 2 vertices each of valence $n$, and all edges with two distinct endpoints. The $n$ edges are all geodesic arcs of the same length equal to the distance between the vertices, and are disjoint away from the vertices. There might be more than one representative of the homotopy class of these graphs satisfying the defining constraints, in which case the $n$-pod is a connected component of a homotopy class, where the homotopies are assumed to be through 1-parameter families satisfying the defining constraints. If the constrained homotopy class defining an $n$-pod consists of a single point, the $n$-pod will be called rigid.
\end{defn}

An $n$-pod is assumed to inherit a marking from the marking on the hyperbolic surface in which it is embedded. The reason for defining $n$-pods is that rigid $n$-pods will be used in the study of the diameter function in much the same way as simple closed curves are used in the study of the systole function. 

\begin{defn}[Local finiteness]
The systoles in $\mathcal{T}_{g}$ have the property that for any point $x\in \mathcal{T}_{g}$ with set $C$ of systoles, there exists a neighbourhood $\mathcal{N}$ of $x$ in $\mathcal{T}_{g}$ in which the systoles are all contained in the set $C$.
\end{defn}

We will need a property like this for the diameter function, with $n$-pods in place of closed curves.

A finite graph will now be defined, related to the diameter locus. This graph will be used for computations with $-f_{\mathrm{dia}}$ that require an analogue of the local finiteness of $f_{\mathrm{sys}}$. The diameter locus might contain infinitely many pairs of points. The motivation for the next graph is to choose finitely many pairs that are sufficiently representative for our applications.  

\begin{defn}[Diameter Graph $\mathcal{G}_{d}(x)$]
When $\mathcal{L}_{d}(x)$ is 0-dimensional, the vertices of $\mathcal{G}_{d}(x)$ are the points of $\mathcal{L}_{d}(x)$, and the edges are geodesic arcs on $S_{g}(x)$ of length equal to $f_{\mathrm{dia}}(x)$ connecting pairs of points in $\mathcal{L}_{d}(x)$.

When $\mathcal{L}_{d}(x)$ has dimension 1,  homotopy classes of the geodesic arcs connecting pairs of points in $\mathcal{L}_{d}(x)$ are assumed to be with respect to homotopies that keep the endpoints of the arcs in $\mathcal{L}_{d}(x)$. If necessary, subdivide the edges of $\mathcal{L}_{d}(x)$ such that every homotopy class of arc has a representative of length $f_{\mathrm{dia}}(x)$ and with both endpoints on vertices of $\mathcal{L}_{d}(x)$. Then the vertices of $\mathcal{G}_{d}(x)$ consist of the vertices of the subdivided diameter locus, and the edges are the geodesics arcs of length $f_{\mathrm{dia}}(x)$ connecting pairs of vertices.

It will be seen that $\mathcal{L}_{d}(x)$ cannot have dimension greater than 1.
\end{defn}

When we say an $n$-pod is \textit{contained in} $\mathcal{G}_{d}(x)$ this means that the homotopy class contains a representative with equal edge lengths, both vertices in $\mathcal{G}_{d}(x)$ and all edges have length $f_{\mathrm{dia}}(x)$ and are contained in homotopy classes corresponding to edges of $\mathcal{G}_{d}(x)$. 


\section{Local finiteness}
\label{secproofs}

\subsection{Rigidity of $n$-pods}
\label{rigiditysec}

We first study 3-pods, showing existence. Perhaps surprisingly, the edge length of a representative of a 3-pod cannot determine the diameter of the hyperbolic surface in which it is embedded unless it is a subgraph of a 4-pod. It will be explained how to construct continuous families of 3-pods to find those contained in $m$-pods with $m\geq 4$. 

\textbf{The symmetric 3-pod.} A symmetric graph is constructed as follows: Cut a 3-holed sphere with geodesic boundary along the shortest geodesic arcs $a_{1}$, $a_{2}$ and $a_{3}$ connecting pairs of geodesic boundary components to obtain a pair of isometric right-angled hexagons. Embed one of the hexagons in the hyperbolic plane, and extend the boundary arcs $a_{1}$, $a_{2}$ and $a_{3}$ into infinite geodesics. As shown in Figure \ref{centroid}, there is a unique circle intersecting each of these geodesics nontransversely in a single point. The center of the circle is contained in the hexagon, and is one of the vertices of our symmetric graph. The other vertex is obtained similarly. Note that in this symmetric example, the angles between a pair of edges are the same at both vertices.

\begin{figure}
\centering
\includegraphics[width=0.5\textwidth]{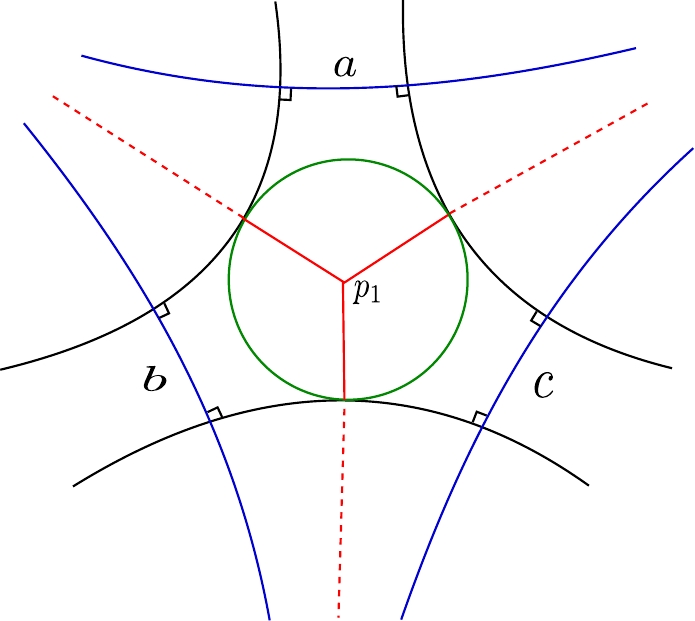}
\caption{Constructing a symmetric representative of a 3-pod.}
\label{centroid}
\end{figure}

\begin{rem}
\label{tripodrem}
A hyperbolic 3-holed sphere with geodesic boundary is determined up to isometry by the lengths $a$, $b$ and $c$ of the boundary components. When $ps_{3}$ is the symmetric 3-valence graph constructed above, it is the unique symmetric representative of a 3-pod, and we can define its length, $L(ps_{3})$, to be the length of its edges. Hyperbolic trigonometry gives
\begin{equation*}
L(ps_{3}) = \tanh^{-1} \sqrt{\frac{\left(\dfrac{\cosh \dfrac{c}{2} + \cosh \dfrac{a}{2} \cosh \dfrac{b}{2}}{\sinh \dfrac{a}{2} \sinh \dfrac{b}{2}} \right)^2- 1}{2 \coth \frac{a}{4} \coth \frac{b}{4} \dfrac{\cosh \dfrac{c}{2} + \cosh \dfrac{a}{2} \cosh \dfrac{b}{2}}{\sinh \dfrac{a}{2} \sinh \dfrac{b}{2}} - \coth^2 \frac{a}{4} - \coth^2 \frac{b}{4}}}
\end{equation*}
\end{rem}

\textbf{Deformability of 3-pods.} Choose two edges of the symmetric graph $ps_{3}$ and call them $e_{1}$ and $e_{2}$. The vertices will be denoted by $p_{1}$ and $p_{2}$. Suppose that concatenating the edges $e_{1}$ and $e_{2}$ gives a curve homotopic to the geodesic $a$ on the boundary of the 3-holed sphere containing the graph. Shift the vertex $p_{1}$ along a path $\gamma_{1}(t)$ whose tangent vector at every point makes equal angles with the edges $e_{1}$ and $e_{2}$. The parameter $t$ measures distance with respect to the hyperbolic metric. Suppose for $t>0$ this deformation increases the length of the edges $e_{1}$ and $e_{2}$, i.e. as shown in Figure \ref{deffig}, $\dot{\gamma}_{1}(t)$ is to the side of $e_{1}$ and $e_{2}$ with angle greater than $\pi$. The point $p_{2}$ is similarly shifted along a path $\gamma_{2}(t)$ whose tangent vector at every point makes equal angles with the edges $e_{1}$ and $e_{2}$. In addition we assume that for $t>0$ this deformation decreases the length of the edges $e_{1}$ and $e_{2}$. As shown in Figure \ref{deffig}, the paths $\gamma_{1}$ and $\gamma_{2}$ are contained in geodesics orthogonal to $a$. 

\begin{figure}
\centering
\includegraphics[width=0.8\textwidth]{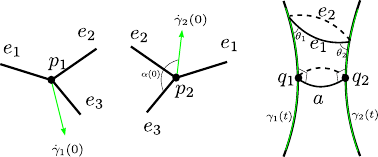}
\caption{The paths $\gamma_{1}$ and $\gamma_{2}$ along which the vertices $p_{1}$ and $p_{2}$ are moved.}
\label{deffig}
\end{figure}

To preserve the constraint that all three edges have the same length, $p_{2}$ is shifted along $\gamma_{2}$. This is possible because at $t=0$, $\dot{\gamma}_{2}$ makes an angle $\alpha(0)$ with $e_{3}$ as shown in Figure  \ref{deffig}, where $\frac{\pi}{2}<\alpha(0)<\pi$, and $\alpha(t)$ increases towards $\pi$ as $t$ increases. 

The different choices of the pair of edges $\{e_{1}, e_{2}\}$ each give different 1-parameter families of deformations; three in total.

For convenience, from now on the same symbol will be used for a  closed geodesic or edge and its length.

\begin{lem}
\label{calclem}
Let $a$ be the length of the closed geodesic in the homotopy class of curves containing the concatenation of the edges $e_{1}$ and $e_{2}$ and $\theta_{1}$ and $\theta_{2}$ the angles shown in Figure \ref{deffig}. Then
\begin{enumerate}
\item{The length of $e_{1}$ increases when both $\theta_{1}$ and $\theta_{2}$ are decreased.}
\item{When the sum of $\theta_{1}$ and $\theta_{2}$ is held constant, the only critical point of the length of $e_{1}$ occurs when $\theta_{1}=\theta_{2}$. This is a local minimum.}
\end{enumerate}
\end{lem}
\begin{proof}
Theorem 3.5.7 of \cite{Ratcliffe} states that 
\begin{equation*}
\cosh(e_{1})=\frac{\cos(\theta_{1})\cos(\theta_{2})+\cosh(\frac{a}{2})}{\sin(\theta_{1})\sin(\theta_{2})}.
\end{equation*}

The lemma is then proven by computing derivatives.
\end{proof}

Typically, for a graph representative of a 3-pod $P_{3}$ with vertices $p_{1}$ and $p_{2}$, it is possible to deform the vertices on a neighbourhood $B(P_{3})$ of $(p_{1}, p_{2})$ in $S_{g}\times S_{g}$ to obtain new graph representatives of the 3-pod. If the vertices cannot be deformed at all, the 3-pod is a subgraph of a rigid $m$-pod, with $m\geq 4$. The boundary of $B(P_{3})$ is reached when further deformation in a given direction does not result in a graph whose edge lengths are equal to the distance between the vertices.

\begin{cor}
\label{defcor}
Suppose the vertices of a graph representative $p_{0}$ of a 3-pod $P_{3}$ are in the interior of $B(P_{3})$. Just as for a non rigid symmetric graph, there are three 1-parameter families of representative graphs, on which the lengths of the edges are increasing away from $p_{0}$.
\end{cor}
\begin{proof}
1-parameter families can be constructed in the same way as for the symmetric graph. It remains to show that such deformations increase edge length. By Lemma \ref{calclem}, when the condition $\frac{\pi}{2}<\alpha(0)$ (the angle $\alpha(0)$ is shown in Figure \ref{deffig}) is satisfied, it follows that the edge length is increasing away from $p_{0}$ in the 1-parameter family. This also increases the lengths of all the edges when $\alpha(0)<\frac{\pi}{2}$ and the angle between $e_{3}$ and $\dot{\gamma}_{1}$ is greater than $\frac{\pi}{2}$.In case $\alpha(0)<\frac{\pi}{2}$ and the angle between $e_{3}$ and $\dot{\gamma}_{1}$ is less than $\frac{\pi}{2}$, reversing the direction of theis path gives the required path in which the lengths are all increasing.

When $e_{3}$ makes angle $\pi/2$ with one or both of $\gamma_{1}$ or $\gamma_{2}$, we need to consider two distinct possibilities, as shown in Figure \ref{deffig2}. In the first case it follows from trigonometric arguments that it is possible to define the 1-parameter family along which the lengths of the edges all increase in the direction of increasing $t$, although this increase will not be to first order at a point at which $e_{3}$ is orthogonal to both $\gamma_{1}$ and $\gamma_{2}$ as indicated on the left side of Figure \ref{deffig2}. This is done in the usual way by choosing $p_{1}$ to be the vertex with the smaller angle between $e_{1}$ and $e_{2}$, and $p_{2}$ the angle with the larger angle between $e_{1}$ and $e_{2}$. In the second case, we need to choose $p_{1}$ to be the vertex with the larger angle between $e_{1}$ and $e_{2}$ and $p_{2}$ the vertex with the smaller angle between $e_{1}$ and $e_{2}$. This is shown on the right of Figure \ref{deffig2}. In this case, we see that $p_{1}$ is shifted a larger distance along $\gamma_{1}$ than $p_{2}$ along $\gamma_{2}$.
\end{proof}

\begin{figure}
\centering
\includegraphics[width=0.8\textwidth]{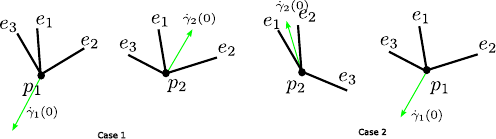}
\caption{The two cases in the proof of Corollary \ref{defcor}.}
\label{deffig2}
\end{figure}

\begin{cor}
If the vertices cannot be deformed at all, the 3-pod is a subgraph of a rigid $m$-pod, with $m\geq 4$.
\end{cor}
\begin{proof}
If a rigid 3-pod were not contained in a 4-pod, every edge $e_{4}$ connecting the vertices of the 3-pod would have length larger than the lengths of the edges of the 3-pod by an amount at least equal to some fixed $\epsilon>0$ not depending on $e_{4}$. This is a consequence of the finiteness result of Lemma \ref{curvelength}. Consequently, the 3-pod can be deformed by a small amount along any of the 1-parameter families from Corollary \ref{defcor} without violating the constraint that the edge length must be equal to the distance between the vertices.
\end{proof}

\begin{rem}
\label{defrem}
From Corollary \ref{defcor} we see that the deformation spaces of a fixed homotopy type of 3-pods with constant edge lengths are at most 1-dimensional. These deformations are obtained as a deformation along one of the 1-parameter families from Corollary \ref{defcor}, followed by a deformation backwards along another.
\end{rem}

The boundary of the set $B(P_{3})\subset S_{g}\times S_{g}$ contains pairs of vertices of compact 1-parameter families of 4-pods. Generically we expect that the edge lengths are not constant in these 1-parameter families. Similar to a linear programming problem, the maximal edge length of $P_{3}$ evaluated over $B$ occurs at the vertices. This motivates the next definition.

\begin{defn}[$n$-pod length]
Denote by $B(P_{n})(x)$ the compact, connected subset of $S_{g}\times S_{g}$ consisting of vertices of graph representatives of an $n$-pod $P_{n}$. The length $L(P_{n})(x)$ of $P_{n}$ at a point $x\in \mathcal{T}_{g}$ for which $B(P_{n})(x)$ is nonempty is given by the supremum of the edge length evaluated over vertices in $B(P_{n})(x)$.
\end{defn}

Since $B(P_{n})(x)$ is compact, the supremum in the definition of $L(P_{n})(x)$ is realised.

\subsection{The diameter locus}
\label{diameterlocussub}

\begin{prop}[Proposition \ref{DLL} from the Introduction]
For any $x\in \mathcal{T}_{g}$, $\mathcal{L}_{d}(x)$ is a smoothly embedded finite graph of dimension at most one embedded in the marked hyperbolic surface $S_{g}$ corresponding to $x$.
\end{prop}
\begin{proof}
This is a consequence of Remark \ref{defrem}. Since the deformation spaces of 3-pods with constant edge lengths are at most 1-dimensional, the deformation spaces of $m$-pods with $m>3$ cannot have dimension greater than 1. Smoothness of any edges of $\mathcal{L}_{d}$ also follows from the explicit description of the deformation spaces from Remark \ref{defrem}.
\end{proof}

\begin{lem}
\label{LFL}
For every point $x\in \mathcal{T}_{g}$ there is a neighbourhood $\mathcal{N}(x)$ in $\mathcal{T}_{g}$ for which $f_{\mathrm{dia}}$ is obtained as
\begin{equation}
\label{eqn1}
\text{-}f_{\mathrm{dia}}(y)=\inf\{\text{-}L(P_{m})\ |\ P_{m}\text{ is an }m\text{-pod contained in }\mathcal{G}_{d}(x)\}
\end{equation}
for every $y$ in $\mathcal{N}(x)$. Moreover, the set of $m$-pods, $m\geq 4$, contained in $\mathcal{G}_{d}(x)$ is finite.
\end{lem}
\begin{proof}
Recall that an $n$-pod $P_{n}$ whose domain contains $x\in \mathcal{T}_{g}$, determines a set of $n(n-1)/2$ curves, each with length less than $2f_{\mathrm{dia}}(x)$ on $S_{g}$. Consequently, by the same argument as in the proof of Lemma \ref{curvelength}, the set of $m$-pods with $m\geq 4$ contained in $\mathcal{G}_{d}(x)$ is finite.

If $P_{n}$ is not contained in $\mathcal{G}_{d}(x)$, the length of $P_{n}$ is less than $f_{\mathrm{dia}}(x)$. The same finiteness as in the first part of the proof then implies that there is an $\epsilon_{1}>0$ such that for every $n$-pod not contained in $\mathcal{G}_{d}(x)$ but whose domain contains $x$ satisfies $L(P_{n})(x)<f_{\mathrm{dia}}(x)-\epsilon_{1}$.

Now choose an $\epsilon>0$ and consider the set $\sigma$ of $n$-pods for which $x$ is not in the domain, but whose length is equal to $f_{\mathrm{dia}}(y)$ at a point $y\in \mathcal{T}_{g}$ within an $\epsilon$-neighbourhood of $x$. Within this $\epsilon$-neighbourhood, $f_{\mathrm{dia}}$ is bounded, and for every point there is a map of the corresponding hyperbolic surface to $S_{g}$ that distorts distances by no more than some bounded factor. Consequently, the homotopy class of graphs representing an element $P_{n}$ of $\sigma$ determines a set of $n(n-1)/2$ curves on $S_{g}$ of length less than $f_{\mathrm{dia}}(x)+\epsilon_{2}$, where $\epsilon_{2}$ can be chosen to work for every element of $\sigma$. Hence $\sigma$ is a finite set. For every element $P_{n}$ of $\sigma$, the domain of $P_{n}$ is closed, and hence has Hausdorff distance from $x$ in $\mathcal{T}_{g}$ bounded from below by some $\epsilon_{3}>0$.
\end{proof}

\begin{prop}[Proposition \ref{PWS} of the Introduction]
$f_{\mathrm{dia}}$ is a piecewise smooth function $\mathcal{T}_{g}\rightarrow \mathbb{R}_{+}$.
\end{prop}
\begin{proof}
 The length of an $m$-pod $P_{n}$ varies smoothly in the interior of its domain, except where the vertex of $B(P_{n})$ representing the maximiser changes, where the length might only vary continuously. On the boundary of a set on which the infimum in Equation \eqref{eqn1} is realised by a fixed $m$-pod,  $f_{\mathrm{dia}}$ might only be continuous. The proposition then follows from Lemma \ref{LFL}.
\end{proof}



\section{Convexity}
\label{tripodsec}
The purpose of this section is to prove the main theorem, namely that $f_{\mathrm{dia}}$ is a topological Morse function. By a result of \cite{ConvexTMF}, this amounts to showing that $n$-pod lengths, where defined on $\mathcal{T}_{g}$, satisfy a convexity property analogous to but weaker than convexity of length functions on $\mathcal{T}_{g}$. The argument is a generalisation of the result in \cite{NRP} that length functions are convex along earthquake paths.

Geodesic laminations are defined and studied in Chapters 8 and 9 of \cite{Th}. A left earthquake path $\mathcal{E}_{\nu}:\mathbb{R}\rightarrow \mathcal{T}_{g}$ is a bi-infinite path generalising the notion of a 1-parameter family of twist deformations around a multicurve to a shearing deformation determined by a measured geodesic lamination $\nu$. This is defined via a limiting process, the details of which are given in Section 2 of \cite{NRP}.

\begin{thm}[Thurston, proof given in the appendix of \cite{NRP}]
\label{th1}
For every $x$ and $y$ in $\mathcal{T}_{g}$, there is a unique left earthquake path from $x$ to $y$.
\end{thm}

A function $f$ is \textit{convex} along $\mathcal{E}_{\nu}$ if for every $t_{0}\in \mathbb{R}$ and every $t\in (0,1)$ we have
\begin{equation*}
f\circ \mathcal{E}_{\nu}(t+t_{0})\leq t f\circ \mathcal{E}_{\nu}(t_{0})+(1-t)f\circ \mathcal{E}_{\nu}(t_{0}+1)
\end{equation*}
The function $f$ is strictly convex along $\mathcal{E}_{\nu}$ if strict inequality holds in the above equation for every $t_{0}\in \mathbb{R}$ and every $t\in(0,1)$. Equivalently, $f$ is strictly convex along $\mathcal{E}_{\nu}$ at $t=t_{0}$ if
$$\frac{d^{2}f\circ \mathcal{E}_{\nu}}{dt^{2}}(t_{0})>0 $$

The proof of convexity of length functions along earthquake paths in \cite{NRP} relied heavily on the following proposition.

\begin{prop}[Proposition 3.5 of \cite{NRP}]
\label{Kercklem}
For every intersection of a geodesic $\gamma$ on $S_{g}$ with $\nu$, the angle of intersection $\theta(t)$ measured anticlockwise from $\gamma$ to $\nu$ is strictly decreasing as a function of $t$ along the earthquake path $\mathcal{E}_{\nu}(t)$.
\end{prop}

\begin{thm}[Theorem \ref{mainthm} from the Introduction]
$f_{\mathrm{dia}}$ is a topological Morse function on $\mathcal{T}_{g}$.
\end{thm}
\begin{proof}
The proof is based on the following result from \cite{ConvexTMF}. The word ``convex'' was used in \cite{ConvexTMF} to mean what we are calling strictly convex here.

\begin{thm}[Theorem 1.2 of \cite{ConvexTMF}]
\label{citethm}
Let $F:=\{f_{i}\ |\ i\in N\}$ be a set of functions on an $n$-dimensional topological manifold $M$, and define a function $f:M\rightarrow \mathbb{R}$ with the property that for every $x\in M$, $f:M\rightarrow \mathbb{R}$ is given by $x\mapsto \min\{f_{i}(x)\ |\ f_{i}\in F\}$. Suppose also that for every $x\in M$, there is a neighbourhood $N(x)$ of $x$ on which $f$ can be evaluated as the minimum of a finite subset $F_{x}$ of $F$, where $N(x)$ is contained in the domain of a chart on which the elements of $F_{x}$ map to strictly convex functions on an open set of $\mathbb{R}^{n}$. Then $f$ is a topological Morse function.
\end{thm}

Define the cone of decrease of $f_{\mathrm{dia}}$ at $x\in \mathcal{T}_{g}$ to be the set of tangent vectors in $T_{x}\mathcal{T}_{g}$ that give directions in which $f_{\mathrm{dia}}$ is nonincreasing. In Theorem \ref{citethm} one can replace ``minimum'' by ``maximum'' with the same proof, except now the argument uses the convexity of the sublevel sets to show that at a noncritical point, the cone of decrease of $f_{\mathrm{dia}}$ (instead of the cone of increase) is the tangent cone to the intersection of convex level sets, and hence truly a cone. Theorem \ref{mainthm} then follows from Lemma \ref{LFL} and the construction of a system of coordinates that will now be given.

Let $\Pi$ be the finite set of $m$-pods whose length determines $f_{\mathrm{dia}}$ on $\mathcal{N}(x)$ by Lemma \ref{LFL}.

 We will make use of shear coordinates as defined in Section 3.4 of \cite{Bestvina}. This set of coordinates involves a choice of a maximal geodesic lamination $\lambda$. The conventions used in the definition of the coordinates then uniquely determine a point in $\mathcal{T}_{g}$ that is the origin of the coordinates.  The shear coordinates of a point $x\in \mathcal{T}_{g}$ are a set of $6g-6$ real numbers that measure the amount of shearing along each leaf of $\lambda$ needed to go from the marked hyperbolic surface corresponding to the origin to to the marked hyperbolic surface corresponding to $x$. For a fixed $x\in \mathcal{T}_{g}$, local finiteness implies that it is possible to choose a maximal geodesic lamination $\lambda$ with the property that every leaf of $\lambda$ intersects every element of $\Pi$, and every leaf is disjoint from the vertices of all $m$-pods in $\Pi$.


Fix $x\in \mathcal{T}_{g}$, and suppose $P_{n}\in \Pi$. We set $x=\mathcal{E}_{\nu}(t_{0})$, where  $\mathcal{E}_{\nu}$ is an earthquake path that involves shearing along one of the leaves $\lambda_{i}=\nu$ of $\lambda$, i.e. varying only one of the shear coordinates, which will be denoted by $t$.

Since the earthquake map is an isometry away from the leaves of $\nu$, in this sense the vertices $p_{1}$ and $p_{2}$ are fixed when the $n$-valent graph representing $P_{n}$ at $\mathcal{E}_{\nu}(t_{0})$ is deformed along $\mathcal{E}_{\nu}$. If the vertices are fixed in this sense, the proof of Corollary 3.4 of \cite{NRP} shows that the change in the length $l_{1}$ of edge $e_{1}$ of the $n$-valent graph is either given by Wolpert's twist formula \cite{Wolperttwist} when $\nu$ is a closed geodesic, or by the generalisation of Wolpert's twist formular to laminations, given by
\begin{equation}
\label{Kerckeqn}
\frac{dl_{1}}{dt}|_{t_{0}}=\int_{e_{1}}\cos(\theta)d\nu
\end{equation}
where $d\nu$ is the transverse measure on the measured lamination $\nu$ and $\theta$ is the function measuring the angle anticlockwise from $e_{1}$ to $\nu$. Analogous equations apply to the lengths $l_{2}, \ldots, l_{n}$ of the edges $e_{2}, \ldots, e_{n}$ of the fixed vertex graph representative of $P_{n}$. It follows from Proposition \ref{Kercklem} that each of the lengths $l_{1}, \ldots, l_{n}$ is a convex function along $\mathcal{E}_{\nu}$. The question remains whether convexity of $L(P_{n})\circ \mathcal{E}_{\nu}(t)$ still holds when the vertices move around the surface in such a way that the defining constraints of an $n$-pod are fulfilled. This appears to be true, but to prove the theorem, there are some cases in which we will replace $t$ by a parameter $\tau$ in order to define the desired coordinate chart.

When considering lengths of $n$-pods, we will consider only rigid $n$-pods. This is sufficient because if an $n$-pod $P_{n}$ is not rigid, its length is realised by a rigid $m$-pod, $m>n$, on the boundary of $B(P_{n})$.

Let $S_{g}(t)$ be the marked hyperbolic surface corresponding to $\mathcal{E}_{\nu}(t)$. Denote by $\gamma_{1}$ and $\gamma_{2}$ the paths traced out in $\mathcal{E}_{\nu}(t)\times S_{g}(t)$ by the vertices $p_{1}$ and $p_{2}$ respectively as $t$ varies from $t_{0}-\epsilon$ to $t_{0}+\epsilon$. The rigidity assumption on $n$-pods ensures that the constraints on the edge lengths of the graph uniquely determine the continuously differentiable paths $\gamma_{1}$ and $\gamma_{2}$. More specifically, to first order, the rate of change of the length of the edge $e_{i}\circ \mathcal{E}_{\nu}(t)$ is given by

\begin{equation}
\label{secondordereq}
\frac{dl_{i}}{dt}(t_{0})-g\left(\hat{e}_{i,1}(t_{0}), \dot{\gamma}_{1}(t_{0})\right)-g\left(\hat{e}_{i,2}(t_{0}), \dot{\gamma}_{2}(t_{0})\right)
\end{equation} 

where $g$ is the induced metric on the cartesian product of $\mathcal{E}_{\nu}(t)\times S_{g}(t)$, $\hat{e}_{i,1}(t)$ is the unit vector in $T_{(t,p_{1})}\mathcal{E}_{\nu}(t)\times S_{g}(t)$ given by the inward pointing 1-sided tangent vector to the edge $e_{i}$ at the (fixed) vertex $p_{1}$, and $\hat{e}_{i,2}(t)$ is the unit vector in $T_{(t,p_{2})}\mathcal{E}_{\nu}(t)\times S_{g}(t)$ given by the inward pointing 1-sided tangent vector to the edge $e_{i}$ at the (fixed) vertex $p_{2}$. 


As illustrated schematically in Figure \ref{secondorderfig}, hyperbolic trigonometry implies that the movement of the vertices along the paths $\gamma_{1}$ and $\gamma_{2}$ is such that the angles between the edges of the equal edge length graph and $\dot{\gamma}_{1}(t)$ and $\dot{\gamma}_{2}(t)$ increase with $t$ unless they are zero, in which case they might stay constant. 

\begin{figure}
\centering
\includegraphics[width=0.6\textwidth]{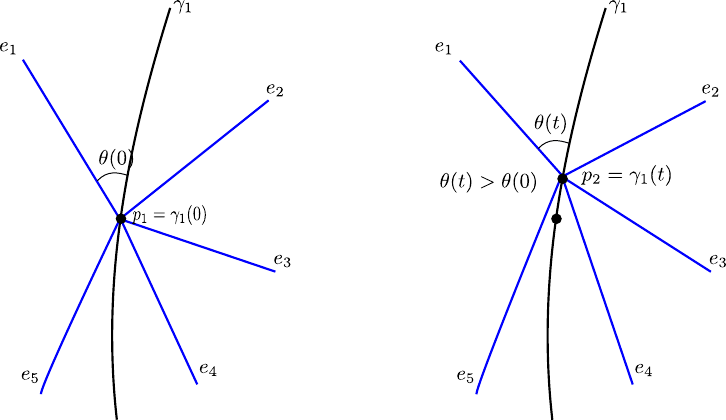}
\caption{Moving the vertex $p_{1}$ along $\gamma_{1}$ in the direction of increasing $t$ increases the angles made by the edges with $\dot{\gamma}_{1}$. The path shown has tangent given by the projection of the vector field $\dot{\gamma(t)}$ onto the tangent space of the fiber.}
\label{secondorderfig}
\end{figure}

For one or both of $|\dot{\gamma}_{1}(t_{0})|$ and $|\dot{\gamma}_{2}(t_{0})|$ to be zero requires the angles of intersection of $\nu$ with the edges of $P_{n}(x)$ to satisfy very specific constraints coming from Equation \eqref{Kerckeqn} and depending on the angles between the edges of $P_{n}(x)$ at the vertices. We choose $\lambda$ such that none of the leaves of $\lambda$ determine an earthquake path with  $\dot{\gamma}_{1}=0$ and/or $\dot{\gamma}_{2}=0$ through the point  $\mathcal{E}_{\nu}(t_{0})$, either for $P_{n}(x)$ or for any of the other rigid $m$-pods in $\Pi$.

In Equation \eqref{secondordereq}, we know that $l_{i}$ is a convex function along $\mathcal{E}_{\nu}(t)$, and that the angle between $\hat{e}_{i,1}$ and $\dot{\gamma}_{1}$ and also the angle between $\hat{e}_{i,2}$ and $\dot{\gamma}_{2}$ is increasing with $t$ or constantly zero. Consequently, $L(P_{n})$ is strictly convex along $\mathcal{E}_{\nu}(t)$ if $|\dot{\gamma}_{1}|$ and $|\dot{\gamma}_{2}|$ are both nondecreasing. If one or both of $|\dot{\gamma}_{1}|$ and $|\dot{\gamma}_{2}|$ is decreasing, and $|\dot{\gamma}_{1}(t_{0})|\neq 0$ and $|\dot{\gamma}_{2}(t_{0})|\neq 0$, write $t=f(\tau)$. It is possible to choose $f$ to be, for example, an exponential function, such that 
\begin{enumerate}
\item{Both $\frac{dt}{d\tau}>0$ and $\frac{d^2t}{d\tau^{2}}>0$ at $t=t_{0}$,}
\item{Both $|\dot{\gamma}_{1}\frac{dt}{d\tau}|$ and $|\dot{\gamma}_{2}\frac{dt}{d\tau}|$ are nondecreasing at $t_{0}$,}
\item{Conditions 1 and 2 hold for all of the other rigid $m$-pods in$\Pi$.}
\end{enumerate}
The parameter $\tau$ is then one of the coordinates defined on the required coordinate chart on a neighbourhood of $x$. The other parameters are obtained from the corresponding sheer coordinate in the same way.
\end{proof}

\bibliography{TMFbib2}
\bibliographystyle{plain}

\end{document}